\documentclass[english]{amsart}
\usepackage[latin1]{inputenc}
\usepackage{amsthm,amsmath,amssymb}
\theoremstyle{plain}
\newtheorem{thm}{Theorem}[section]
\newtheorem{prop}[thm]{Proposition}
\newtheorem{lem}[thm]{Lemma}

\theoremstyle{definition}

\newtheorem{exa}[thm]{Example}
\theoremstyle{remark}

\newcommand{\Sh}{\mathrm{Sh}}
\newcommand{\Fer}{\mathrm{Fer}}
\newcommand{\eqdef}{\stackrel{\rm def}{=}}
\usepackage{babel}

\newcommand{\St}{\mathcal {ST}}

\newcommand{\Irr}{\mathrm{Irr}}
\newcommand{\GCD}{\mathrm{GCD}}
\newcommand{\Inv}{\mathrm{Inv}}
\newcommand{\Id}{\mathrm{Id}}
\newcommand{\inv}{\mathrm{inv}}
\newcommand{\Pair}{\mathrm{Pair}}
\newcommand{\pair}{\mathrm{pair}}

\newcommand{\fix}{\mathrm{fix}}

\newcommand{\Ind}{\mathrm{Ind}}

\newcommand{\Span}{\mathrm{Span}}
\author{Fabrizio Caselli}\author {Roberta Fulci}
\title{Refined Gelfand models for wreath products}
\address{Dipartimento di matematica, Universit\`a di Bologna \\Piazza di Porta San Donato 5 \\ Bologna 40126, Italy}
\begin{document}
\thanks{\emph{MSC:} 05E15}
\keywords{Reflection groups, models, absolute involutions.}
\maketitle
\begin{abstract}
In [F. Caselli, Involutory reflection groups and their models, J. Algebra 24 (2010), 370--393] it is constructed a uniform Gelfand model for all non-exceptional irreducible complex reflection groups which are involutory. This model can be naturally decomposed into the direct sum of submodules indexed by symmetric conjugacy classes, and in this paper we present a simple combinatorial description of the irreducible decomposition of these submodules if the group is the wreath product of a cyclic group with a symmetric group. This is  attained by showing that such decomposition is compatible with the generalized Robinson-Schensted correspondence for these groups.
\end{abstract}

\section{Introduction}

A Gelfand model of a finite group $G$ is a $G$-module isomorphic to the multiplicity-free sum of all its irreducible complex representations. In \cite{Ca2}, a Gelfand model is constructed for all \emph{involutory} reflection groups, a class of finite complex reflection groups which contains all infinite families of irreducible Coxeter groups and all the wreath products $G(r,n)$. More precisely, a finite subgroup $G$ of $GL(n,\mathbb{C})$ will be called involutory if the dimension of its Gelfand model is equal to the number of its \textit{absolute involutions}, i.e.  elements $g$ satisfying $g \bar{g}=1$, where $\bar g$ denotes the entrywise complex conjugate of $g$. If we restrict our attention to complex reflection groups of the form $G(r,p,n)$, we have the following result (\cite[Theorem 4.5]{Ca2}).

\begin{thm} Let $G$ be a complex reflection group of the form $G(r,p,n)$. Then $G$ is involutory if and only if $\GCD(p,n)=1,2$.
\end{thm}

The Gelfand model for $G(r,p,n)$ constructed in \cite{Ca2} is based on the theory of projective reflection groups introduced in \cite{Ca1}. If  $G=G(r,n)$, the setting is much simpler and the model $(M, \varrho)$  for such groups looks as follows.

\begin{itemize}
 \item $M$ is the vector space having a basis indexed by the set $I(r,n)$ of the absolute involutions of $G(r,n)$, :
$$M=\bigoplus_{v \in I(r,n)}\mathbb{C}\,\,C_v$$

\item the morphism $\varrho:G(r,n)\rightarrow GL(M) $ has the form
$$\varrho(g)(v)=\phi_g(v)C_{|g|v|g|^{-1}},$$
where $\phi_g(v)$ is a scalar and $|g|$ is the natural projection of $g\in G(r,n)$ into $S_n$ which ``forgets'' colors.
\end{itemize}

A more precise description of the model (and of the notation used) is deferred to \S\ref{NotPre}. Nevertheless, what we know about this model is already enough to observe that there is an immediate decomposition of $M$ into submodules. To describe this, we need one further definition. If $g, h \in G(r,n)$ we say that $g$ and $h$ are $S_n$-conjugate if there
exists $\sigma \in S_n$ such that $g=\sigma h\sigma^{-1}$, and we call \textit{$S_n$-conjugacy classes}, or \textit{symmetric conjugacy classes}, the corresponding equivalence classes.
If $c$ is a $S_n$-conjugacy class of absolute involutions in $I(r,n)$
we denote by $M(c)$ the subspace of $M$ spanned by the basis elements $C_v$ indexed by the absolute involutions
$v$ belonging to the class $c$, and it is clear that
$$
M=\bigoplus_c M(c)\,\,\,\textrm{ as $G$-modules,}
$$
where the sum runs through all $S_n$-conjugacy classes of absolute
involutions in $I(r,n)$. It is natural to ask if we can describe the
irreducible decomposition of the submodules $M(c)$. This
decomposition is known if $G$ is the symmetric group $S_n=G(1,n)$ (see
\cite{IRS,APR1}). We will show that the irreducible decomposition of these submodules is
well behaved with respect to the generalized Robinson-Schensted correspondence (see \S\ref{NotPre}) introduced and studied by Stanton and White \cite{SW}, so answering to a problem raised in \cite{Ca2} (see also \cite{APR} for an analogous question). 

Let us briefly clarify the meaning of 'well behaved with respect to the Robinson-Schensted correspondence'.

The irreducible representations of $G(r,n)$ are naturally parametrized by the
elements of the set $\Fer(r,n)$, i.e. the set of  $r$-tuples of Ferrers diagrams
$(\lambda^{(0)},\ldots,\lambda^{(r-1)})$ with $\sum |\lambda^{(i)}|=n$ (see Proposition \ref{rapp di grn}), and we denote by $\rho_{\lambda^{(0)},\ldots,\lambda^{(r-1)}}$ the irreducible representation of $G(r,n)$ corresponding to the $r$-tuple $(\lambda^{(0)},\ldots,\lambda^{(r-1)})\in \Fer(r,n)$.
If $v$ is an absolute involution in $G(r,n)$, 
we denote by $\Sh(v)$ the element of $\Fer(r,n)$ which is the shape of the multitableaux of the image of $v$
via the generalized Robinson-Schensted correspondence. Namely, we let
$$\Sh(v)\eqdef(\lambda^{(0)},\ldots,\lambda^{(r-1)}),$$
 where 
$$v\stackrel{RS}{\longrightarrow}[(P_0,\ldots,P_{r-1}),(P_0,\ldots,P_{r-1})], \quad P_i \mbox{ of shape }\lambda^{(i)}.$$
For notational convenience we also let $\Sh(c)=\cup_{v\in c}\Sh(v)\subset \Fer(r,n)$ and we are now ready to state the main result of this work.
\begin{thm}\label{mainth}
Let $c$ be a $S_n$-conjugacy class of absolute involutions in $G(r,n)$. Then
the following decomposition holds:
$$M(c) \cong \bigoplus_{\begin{subarray}{c}
(\lambda^{(0)},\ldots,\lambda^{(r-1)})\in \Sh(c)\\
 \end{subarray}}
  \rho_{\lambda^{(0)},\ldots,\lambda^{(r-1)}}.$$
\end{thm}

For the reader's convenience we will treat the case of the  Weyl groups $B_n\eqdef G(2,n)$ of type $B$ in full detail, and we will describe afterwards in \S\ref{grn} the outline of the proof for the general case of wreath products, focusing in particular on how the proof for $B_n$ should be adapted in this general case. 


The paper is organized as follows. In \S\ref{NotPre} we  collect the notation and the preliminary results which are needed, including a description of the generalized Robinson-Schensted correspondence for wreath products studied by Stanton and White in \cite{SW} and the definition of the Gelfand model for $G(r,n)$ constructed by the first author in \cite{Ca2}. In \S\ref{BnLemmone} we generalize an idea appearing in \cite{IRS} to provide a characterization in terms of inductions and restrictions of a representation of $B_n$ which is the multiplicity-free sum of all irreducible representations indexed by pairs of Ferrers diagrams having all rows of even length. In \S\ref{parz}, which is really the heart of the paper, using the characterization obtained in the previous section, we describe a partial result of Theorem \ref{mainth} which is the irreducible decomposition of the submodule $M(c)$ corresponding to the symmetric conjugacy class $c$ of involutions with no fixed points and with a given number of negative entries. The proof of the full result appears then in \S\ref{full}, where we make use of some results which are the analogous in type $B$ of very well-known facts about symmetric groups.
Finally, in \S\ref{grn}, we sketch a proof of the general result for the wreath products $G(r,n)$.

We end this introduction by mentioning that the more involved case of involutory reflection groups of the form $G(r,p,n)$ with $\GCD(p,n)=1,2$ will be treated by the authors in a forthcoming paper that makes use of the results of the present work. 
\section{Notation and prerequisites}\label{NotPre}

In this section we collect the notation that is used in this paper as well as the preliminary results that are needed.

We let $\mathbb Z $ be the set of integer numbers and $\mathbb N$ be the set of nonnegative integer numbers. For $a,b\in \mathbb Z$, with $a\leq b$ we let $[a,b]=\{a,a+1,\ldots,b\}$ and, for $n\in \mathbb N$ we let $[n]\eqdef[1,n]$. For $r\in\mathbb N$, $r>0$, we let $\mathbb Z_r\eqdef \mathbb Z /r\mathbb Z$ and $\zeta_r$ be the primitive $r$-th root of unity $\zeta_r\eqdef e^{\frac{2\pi i}{r}}$. 

The main subject of this work are the wreath products $G(r,n)=C_r\wr S_n$  that we are going to describe. If $A$ is a matrix with complex entries we denote by $|A|$ the real matrix whose entries are the absolute values of the entries of $A$.
The \emph{wreath product} group $G(r,n)$ can be realized as the group of all $n\times n$ matrices satysfying the following conditions: 
\begin{itemize}
\item the non-zero entries are $r$-th roots of unity;
\item there is exactly one non-zero entry in every row and every column (i.e. $|A|$ is a permutation matrix).
\end{itemize}

%

If the non-zero entry in the $i$-th row of $g\in G(r,n)$ is $\zeta_r^{z_i}$ we let $z_i(g)\eqdef z_i\in \mathbb Z_r$, we say that $z_1(g),\ldots,z_n(g)$ are the \emph{colors} of $g$, and we let $z(g)=\sum z_i(g)$. 

We sometimes think of an element $g\in G(r,n)$ as a \emph{colored permutation}, i.e. as a map
\begin{eqnarray*}
\langle\zeta_r\rangle[n]&\rightarrow &\langle\zeta_r\rangle[n]\\
\zeta_r^k i &\mapsto &\zeta_r^{k+z_i(g)}|g|(i),
\end{eqnarray*}
where $\langle \zeta_r\rangle [n]$ is the set of numbers of the form $\zeta_r^{k}i$ for some $k\in \mathbb Z_r$ and $i\in [n]$, and $|g|\in S_n$ is the permutation defined by $|g|(i)=j$ if $g_{i,j}\neq 0$. We may observe that an element $g\in G(r,n)$ is uniquely determined by the permutation $|g|$ and by the color vector $(z_1(g),\ldots,z_n(g))$, and we will often write $g=[\sigma;z_1,\ldots,z_n]$, with $\sigma\in S_n$ and $z_i\in \mathbb Z_r$ meaning that $|g|=\sigma$ and $z_i(g)=z_i$ for all $i\in [n]$. Sometimes it can also be convenient to make use of the \emph{window notation} of $g$ and write $g=[g(1),\ldots,g(n)]$. 

In \cite[Chapter 4]{JK} we can find a description of the set $\Irr(r,n)$ consisting of all irreducible complex representations of $G(r,n)$ that we briefly recall.
Given a partition $\lambda=(\lambda_1,\ldots,\lambda_l)$ of $n$, the \emph{Ferrers diagram of shape $\lambda$} is a collection of boxes, arranged in left-justified rows, with $\lambda_i$ boxes in row $i$. 
We denote by $\Fer(r,n)$ the set of $r$-tuples $(\lambda^{(0)},\ldots,\lambda^{(r-1)})$ of Ferrers diagrams such that $\sum |\lambda^{(i)}|=n$.  If $\mu\in \Fer(r,n)$ we denote by $\St_{\mu}$ the set of all possible fillings of the boxes in $\mu$ with all the numbers from 1 to $n$ appearing once, in such way that rows are increasing from left to right and columns are incresing from top to bottom in every single Ferrers diagram of $\mu$. We also say that $\St_{\mu}$ is the set of \emph{standard multitableaux} of shape $\mu$. Moreover we let $\St(r,n)\eqdef \cup_{\mu\in \Fer(r,n)} \St_\mu$.\\
 In the following result the set $\Irr(r,n)$ is described explicitly in terms of the irreducible representations of the symmetric group. Here and in what follows we use the symbol $\odot$ for \emph{external} tensor product of representations and the symbol $\otimes$ for \emph{internal} tensor product of representations.
\begin{prop}\label{rapp di grn}
We have
$$
\Irr(r,n)=\{\rho_{\lambda^{(0)},\ldots,\lambda^{(r-1)}},\textrm{ with }(\lambda^{(0)},\ldots,\lambda^{(r-1)})\in \Fer(r,n)\},
$$
where the irreducible representation $\rho_{\lambda^{(0)},\ldots,\lambda^{(r-1)}}$ of $G(r,n)$ is given by
  $$\rho_{\lambda^{(0)},\ldots,\lambda^{(r-1)}}= \Ind_{G(r,n_0)\times\cdots\times G(r,n_{r-1})}^{G(r,n)} \left(\bigodot_{i=0}^{r-1}(\gamma_{n_i}^{\otimes i}\otimes \tilde \rho_{\lambda^{(i)}})\right),$$
where:
\begin{itemize}
        \item $n_i=|\lambda^{(i)}|$;
        \item $\tilde \rho_{\lambda^{(i)}}$ is the natural extension to $G(r,n_i)$ of the irreducible (Specht) representation $\rho_{\lambda^{(i)}}$ of $S_{n_i}$, i.e. $\tilde \rho_{\lambda^{(i)}}(g)\eqdef \rho_{\lambda^{(i)}}(|g|)$ for all $g\in G(r,n_i)$.
        \item $\gamma_{n_i}$ is the
1-dimensional
  representation of $G(r,n_i)$ given by
\begin{align*}
 \gamma_{n_i}:G(r,n_i)&\rightarrow  \mathbb{C}^*\\
 g&\mapsto \zeta_r^{z(g)}.
\end{align*}
      \end{itemize}
\end{prop}

Recall the classical Robinson-Schensted correspondence  from \cite[\S 7.11]{Sta}). 
This correspondence has been extended to wreath product groups $G(r,n)$ in \cite{SW} in the following way. Given $g\in G(r,n)$ and  $j\in \mathbb Z_r$, we let $\{i_1,\ldots,i_h\}=\{l\in [n]: z_l(g)=j\}$, with $i_k<i_{k+1}$ for all $k\in [h-1]$, and we consider the two-line array $A_j=\left(\begin{array}{cccc}i_1&i_2&\cdots&i_h\\ \sigma(i_1)&\sigma(i_2)&\cdots&\sigma (i_h)\end{array}\right)$, where $\sigma=|g|$, and the pair of tableaux $(P_j,Q_j)$ obtained by applying the Robinson-Schensted correspondence to $A_j$. Then the correspondence 
$$
g\mapsto [P(g),Q(g)]\eqdef [(P_0,\ldots,P_{r-1}),(Q_0,\ldots,Q_{r-1})]
$$
is a bijection between $G(r,n)$ and pairs of standard multitableaux in $\St(r,n)$ of the same shape, and we call it the generalized Robinson-Schensted correspondence. We also recall that an element $g\in G(r,n)$ is an absolute involution if and only if $g\mapsto[(P_0,\ldots,P_{r-1}),(P_0,\ldots,P_{r-1})]$ for some $(P_0,\ldots,P_{r-1})\in \St(r,n)$ under the generalized Robinson-Schensted correspondence. 

 If $M$ is a complex vector space and $\rho:G\rightarrow GL(M)$ is a representation of $G$ we say that the pair $(M,\rho)$ is a \emph{Gelfand model} if  it is isomorphic as a $G$-module to the direct sum of all irreducible modules of $G$ with multiplicity one. 

A particular case of the main result in \cite{Ca2} is the explicit construction of a Gelfand model for the groups $G(r,n)$ that we are going to describe.

If $\sigma, \tau \in S_n$ with $\tau^2=1$ we let $\inv_{\tau}(\sigma)=|\{\Inv(\sigma)\cap \Pair(\tau)|$, where 
$$\Inv(\sigma)=\{\{i,j\}:(j-i)(\sigma(j)-\sigma(i))<0\}$$
and
$$
\Pair(\tau)=\{\{i,j\}:\tau(i)=j\neq i\}.
$$
If $g\in G(r,n)$ and $v\in I(r,n)$ we let $$\inv_v(g)=\inv_{|v|}(|g|),$$
and
$$
<g,v>=\sum_{i=1}^n z_i(g)z_i(v) \in \mathbb Z_r.
$$ 
\begin{thm}\label{main}Let  $M\eqdef \bigoplus_{v\in I(r,n)}\mathbb C C_v$ and $\varrho:G(r,n)\rightarrow GL(M)$ be defined by
$$
\varrho(g) (C_v) \eqdef \zeta_r^{<g,v>} (-1)^{\inv_v(g)}  C_{|g|v|g|^{-1}}.
$$
Then $(M,\varrho)$ is a Gelfand model for $G(r,n)$.
\end{thm}
Theorem \ref{main} motivates the following definition. We call the map $v\mapsto |g|v|g|^{-1}$ the \emph{absolute conjugation} by $g$ on $G(r,n)$. This map gives rise to an action of $G(r,n)$ on itself that we still call absolute conjugation.
\section{Some tools in the combinatorial representation theory of $B_n$}\label{BnLemmone}

As mentioned in the introduction, we will now focus our treatment on the special case $B_n=G(2,n)$. 
The main result of this section is Proposition \ref{indres} which is an extension of an idea appearing in \cite{IRS}
and  will be of crucial importance to prove Theorem \ref{mainth}.

First of all we observe that, since $B_n$ is given by real matrices, the absolute involutions in $B_n$ are exactly the involutions in $B_n$. So, to understand our results, we need to describe and parametrize the $S_n$-conjugacy classes of involutions in $B_n$ explicitly.
To this aim, for all $v\in I(2,n)$ we let
 \begin{itemize}
  \item $\fix_0(v)\eqdef|\{i:i>0 $ and $v(i)=i\}|;$
  \item $\fix_1(v)\eqdef|\{i:i>0$ and $v(i)=-i\}|;$
  \item $\pair_0(v)\eqdef|\{(i,j):0<i<j, \,v(i)=j $ and $v(j)=i\}|;$
  \item $\pair_1(v)\eqdef|\{(i,j):0<i<j, \, v(i)=-j$ and $v(j)=-i\}|.$
\end{itemize}
For example, if $v=[(3,2,1,8,9,6,7,4,5);1,0,1,0,1,1,0,0,1]$ which is equivalent, in the window notation, to $v=[-3,2,-1,8,-9,-6,7,4,-5]$, we have $\fix_0(v)=2$, $\fix_1(v)=1$, $\pair_0(v)=1$ and $\pair_1(v)=2$.
\begin{prop}\label{Sn-classes}
Two involutions $v$, $w$ of $B_n$ are $S_n$-conjugate if and only
if
$$\fix_0(v)=\fix_0(w), \quad \pair_0(v)=\pair_0(w),$$ $$\fix_1(v)=\fix_1(w), \quad \pair_1(v)=\pair_1(w).$$
Furthermore, given an involution $v$ in $B_n$, let $\Sh(v)=(\lambda,
\mu)$. Then $\lambda$ has $\fix_0(v)$ columns of odd length and $\fix_0(v) +
2\,\pair_0(v)$ boxes, while $\mu$ has $\fix_1(v)$  columns of odd length and
$\fix_1(v) + 2\,\pair_1(v)$ boxes.
\end{prop}
\begin{proof}
The first part is clear, since conjugation of a cycle by an element in $S_n$ does not alter the number of negative entries in the cycle.
The second part follows easily from the corresponding result for the symmetric group due to Sch\"utzenberger (see \cite{Sc} or \cite[Exercise 7.28]{Sta}) and the definition of the generalized Robinson-Schensted correspondence given in \S\ref{NotPre}.
\end{proof}

We can thus name the $S_n$-conjugacy classes of the involutions of
$B_n$ in this way:
$$c_{f_{0},f_{1},p_{0},p_{1}}\eqdef\{v\in I(2,n): \fix_0(v)=f_{0};\, \fix_1(v)=f_{1};\, \pair_0(v)=p_{0};\, \pair_1(v)=p_{1}\},$$
where $f_{0},f_{1},p_{0},p_{1}\in \mathbb N$ are such that $f_{0}+f_{1}+2p_{0}+2p_{1}=n$.
The description given of the $S_n$-conjugacy classes ensures that
the subspace of $M$ generated by the involutions $v \in B_n$ with
$\fix_0(v)=\fix_1(v)=0$ - which is non trivial if $n$ is even only - is a $B_n$-submodule. The crucial step in
the proof of Theorem \ref{mainth} is the partial result regarding
this submodule.


Given $\lambda\in \Fer(n)$ we let
\begin{eqnarray*}
R^-_{\lambda}&\eqdef&\{\sigma\in \Fer(n-1):\sigma \textrm{ is obtained by deleting one box from }\lambda\}\\
R^+_{\lambda}&\eqdef&\{\sigma\in \Fer(n+1):\sigma \textrm{ is obtained by adding one box to }\lambda\}
\end{eqnarray*}
Moreover, if $(\lambda,\mu)\in \Fer(2,n)$, we let
\begin{eqnarray*}
R^-_{\lambda,\mu}&\eqdef&\{(\sigma,\mu)\in \Fer(2,n-1):\sigma\in R^-_{\lambda}\}\cup \{(\lambda,\tau)\in \Fer(2,n-1):\tau\in R^-_{\mu}\}\\ 
R^+_{\lambda,\mu}&\eqdef&\{(\sigma,\mu)\in \Fer(2,n+1):\sigma\in R^+_{\lambda}\}\cup \{(\lambda,\tau)\in \Fer(2,n+1):\tau\in R^+_{\mu}\}
\end{eqnarray*}

We always identify $B_{n}$ as a subgroup of $B_{n+1}$ by
$$B_{n}=\{g \in B_{n+1}: g(n+1)=n+1\}.$$
\begin{thm}{(Branching rule for $B_n$)}
Let $(\lambda, \mu)\in \Fer(2,n)$. 

Then the following holds:
$$\rho_{\lambda, \mu}\downarrow_{B_{n-1}}=\bigoplus_{(\sigma, \tau) \in \,
R^-_{\lambda, \mu}}  \rho_{\sigma, \tau}$$
$$\rho_{\lambda, \mu}\uparrow^{B_{n+1}}=\bigoplus_{(\sigma, \tau) \in \,
 R^+_{\lambda, \mu}} \rho_{\sigma, \tau}.$$ 
\end{thm}
\begin{proof}
 See \cite[\S3]{hiss}.
\end{proof}
Before stating the main result of this section we need some further notation. 
A diagram $(\lambda, \mu)\in \Fer(2,n)$ will be called \textit{even} if both $\lambda$ and $ \mu$ have all rows of even length.\\
 If $\phi$ and $\psi$ are representations of a group $G$, we say that $\phi$ \emph{contains} $\psi$ if $\psi$ is isomorphic to a subrepresentation of $\phi$.

\begin{prop} \label{indres}
 
 Let $\Pi_m$ be representations of $B_{2m}$, $m$ ranging in $\mathbb{N}$.
 Then the following are equivalent:
 \begin{enumerate}
         \item [a)] for every $m$, $\Pi_m$ is isomorphic to the direct sum of all the irreducible representations of
         $B_{2m}$ indexed by even diagrams of $\Fer(2,2m)$, each of such representations occurring once;
         \item [b)] for every $m$,
           \begin{enumerate}
               \item[b0)] $\Pi_0$ is 1-dimensional (and $B_0$ is the group with 1 element);
               \item[b1)] the module $\Pi_m$
               contains the irreducible representations $\rho_{\iota_{2m}, \emptyset}$ and $\rho_{\emptyset, \iota_{2m}}$ of $B_{2m}$, where $\iota_k$ denotes the single-rowed Ferrers diagram with $k$ boxes;
               \item [b2)]the following isomorphism holds:
               \begin{equation}\label{resind}
                  \Pi_m\downarrow_{B_{2m-1}}\cong\Pi_{m-1}\uparrow^{B_{2m-1}}.
               \end{equation}

           \end{enumerate}
 \end{enumerate}
\end{prop}

We explicitely observe that we are dealing here with even diagrams, i.e., with \textit{rows} of even length. What we will need later are  diagrams with \textit{columns} of even length. This is a harmless difference which simplifies our computations and will be solved in  \S\ref{parz}.
\begin{proof}
 a)$\Rightarrow$ b). Conditions b0) and b1) follow immediately.

 Let us now  compare $\Pi_m\downarrow_{B_{2m-1}}$ and $\Pi_{m-1}\uparrow^{B_{2m-1}}$.
 The branching rule ensures that $\Pi_m\downarrow_{B_{2m-1}}$ contains exactly the
 $\rho_{\lambda, \mu}$'s where the diagram $(\lambda,\mu)$ has exactly one row of odd length. Furthermore, the pair $(\alpha, \beta)$ such that $R_{\alpha,\beta}^-\ni (\lambda, \mu)$ is uniquely determined: to obtain it, it will only be allowed
 to add a box to the unique odd row of the diagram $(\lambda,\mu)$. This means that
 $\Pi_m\downarrow_{B_{2m-1}}$ is the multiplicity-free direct sum of all the
 representations of $B_{2m-1}$ indexed by diagrams in $\Fer(2,2m-1)$ with exactly one row of odd length.

 Arguing analogously for $\Pi_{m-1}\uparrow^{B_{2m-1}}$, we can infer that it
 contains exactly the same irreducible representations with multiplicity 1
 and it is thus isomorphic to $\Pi_m\downarrow_{B_{2m-1}}$.

 \bigskip

 b)$\Rightarrow$ a) Let us argue by induction.

 \bigskip

 The case $m=0$ is given by b0). Let's see also the case $m=1$. We know that $\Pi_1\downarrow_{B_{1}}\cong\Pi_0\uparrow^{B_{1}}\cong\rho_{\iota_1,\emptyset} \oplus \rho_{\emptyset, \iota_1}.$
 But $\Pi_1$ contains $\rho_{\iota_2,\emptyset}$ and $ \rho_{\emptyset, \iota_2}$ by b1), and the isomorphism
 $$\big(\rho_{\iota_2,\emptyset} \oplus \rho_{\emptyset, \iota_2}\big)\downarrow_{B_1}\cong\rho_{\iota_1,\emptyset} \oplus \rho_{\emptyset, \iota_1}\cong\Pi_0\uparrow^{B_{1}}$$
 ensures that $$\Pi_1\cong\rho_{\iota_2,\emptyset} \oplus \rho_{\emptyset, \iota_2}.$$

 \bigskip

 Let us show that, if $\Pi_{m-1}$ is the direct sum of all the
 representations indexed by even diagrams, the same holds for $\Pi_m$.
 For notational convenience, we let 
$$
\Lambda_m\eqdef\{(\lambda,\mu)\in \Fer(2,2m):\, \rho_{\lambda,\mu}\textrm{ is a subrepresentation of }\Pi_m\}
$$First we shall see that, if $(\lambda, \mu)\in \Fer(2,2m)$ is an even diagram, then 
 $(\lambda, \mu)\in \Lambda_m$.
 
 The set $\Fer(2,2m)$
 is totally ordered in this way:
 given two pairs $(\lambda, \mu), (\sigma, \tau)\in \Fer(2,2m)$,
  we let $(\lambda, \mu)<(\sigma, \tau)$ if one of the following holds:

 i) $\lambda<\sigma$ lexicographically;

 ii) $\lambda=\sigma$ and $\mu<\tau$ lexicographically.

 \medskip

 We observe that $(\iota_{2m}, \emptyset)$ is the maximum element of $\Fer(2,2m)$ with respect to this order.
\medskip

We claim that if $(\lambda, \mu)\in \Fer(2,2m)$ is such that:
  \begin{enumerate}
   \item[i)] $(\lambda, \mu)$ is even;
   \item[ii)] $(\lambda, \mu) \notin \{(\iota_{2m}, \emptyset),(\emptyset, \iota_{2m})\}$;
   \item[iii)] $(\sigma, \tau)\in \Lambda_m \textrm{ for all }(\sigma, \tau)\in \Fer(2,2m)\textrm{ such that }\,(\sigma, \tau)
   \mbox{ is even and }(\sigma, \tau)>(\lambda, \mu).$
   \end{enumerate}
   Then $(\lambda, \mu) \in \Lambda_m.$

 As we already know that $(\iota_{2m}, \emptyset)$ and $(\emptyset, \iota_{2m})$ are contained in $\Lambda_m$,
 once proved the claim, all the even pairs will.\\
\emph{Proof of the claim.}  Let $(\lambda, \mu)\in \Fer(2,2m)$ be an even diagram satisfying i), ii) and iii).
  Then the pair $(\lambda, \mu)$ has at least two rows. We let $(\sigma,\tau)\in \Fer(2,2m)$ be the pair obtained from $(\lambda, \mu)$ by deleting two boxes in the last non-zero row and adding two boxes to the first non-zero row. 

  As $(\sigma, \tau)>(\lambda, \mu)$, we have $(\sigma, \tau)\in \Lambda_m$, so the isomorphism
 \eqref{resind}, the induction hypothesis and the branching rule lead to the following:
  \begin{equation}\label{formalissima}
 \forall \,\,(\eta, \theta) \in R^-_{\sigma, \tau},\,\,R^+_{\eta,\theta}\cap \Lambda_m=\{(\sigma,\tau)\}.
  \end{equation}

  Now let $(\alpha,\beta)\in \Fer(2,2m-1)$ be obtained from $(\lambda,\mu)$ by deleting one box in the last nonzero row. Our induction hypothesis ensures that $\rho_{\alpha,\beta}$ is a subrepresentation of $\Pi_{m-1}\uparrow^{B_{2m-1}}$ with multiplicity one. So the isomorphism $\eqref{resind}$ implies that  
   \begin{equation}\label{formalona}
  \textrm{ there exists a unique  $(\gamma,\delta)\in \Fer(2,2m)$ such that } \{(\gamma,\delta)\}=R^+_{\alpha,\beta}\cap \Lambda_m.
  \end{equation}
The claim will be proved if we show that $(\gamma,\delta)=(\lambda,\mu)$.

 The pair $(\gamma,\delta)$ is obtained from  $(\alpha,\beta)$ by adding a single box, since $(\gamma,\delta)\in R^+_{\alpha,\beta}$. If such box is not added in the first or in the last non zero rows of $(\alpha,\beta)$ then $(\gamma,\delta)$ has two rows of odd length and one can check that $R^-_{\gamma,\delta}$ contains at least a diagram with three rows of odd length. This contradicts \eqref{resind}. 

  Now assume that $(\gamma,\delta)$ is obtained by adding a box in the first nonzero row of $(\alpha,\beta)$. If we let $(\eta,\theta)$ be the pair obtained from $(\lambda,\mu)$ by deleting two boxes in the last nonzero row and adding one box in the first nonzero row, we have $(\eta,\theta)\in R^-_{\sigma,\tau}$, and $R^+_{\eta,\theta}\cap \Lambda_m\supseteq\{(\sigma,\tau),(\gamma,\delta) \}$ which 
contradicts \eqref{formalissima}.

Therefore $(\gamma,\delta)$ is obtained by adding a box in the last nonzero row of $(\alpha,\beta)$, i.e. $(\gamma,\delta)=(\lambda,\mu)$ and the claim is proved.

 We have just proved that if we let $\Pi_m^{\mathrm{even}}$ be the multiplicity free sum of all irreducible representations of $B_{2m}$ indexed by even diagrams we have that $\Pi_m^{\mathrm{even}}$ is a subrepresentation of $\Pi_m$. The result follows since we also have 
 $$\Pi_m^{\mathrm{even}}\downarrow_{B_{2m-1}}\cong\Pi_{m-1}\uparrow^{B_{2m-1}},$$
and so, in particular, $\dim(\Pi_m^{\mathrm{even}})=\dim(\Pi_m)$.
 
\end{proof}

\section{A partial result for $B_n$}\label{parz}
In the process of proving our main results we use the following auxiliary representation of $B_n$ on $M$:
\begin{align*}
 \varphi(g): \,&M \rightarrow M\\
 & \, C_v \mapsto (-1)^{<g,v>}C_{|g|v|g|^{-1}}.
\end{align*}Notice that the representation $\varphi$ is just like the
representation $\varrho$ of the model $(M, \varrho)$, apart from the factor
$(-1)^{\inv_v(g)}$.

Let $M_m$ be the subspace of $M$ spanned by the elements $C_v$ as $v$ varies among all
involutions in $B_{2m}$ such that
$\fix_0(v)=\fix_1(v)=0$:$$M_m\eqdef\bigoplus_{p_0+p_1=m}
M(c_{0,0,p_{0},p_{1}}).$$ The main task of this section is to show that the representations $(M_m,\varphi)$ satisfy the conditions of Proposition \ref{indres}.

We first prove that the representation $(M_m,\varphi)$ satisfies condition b1) of Proposition \ref{indres}. In fact, we will show explicitly that $(M_m,\varphi)$ contains all irreducible representations indexed by a pair of 1-rowed Ferrers diagrams.\\ 
Recall from Proposition \ref{rapp di grn} that the irreducible representations of $B_n$  are parametrized by
pairs $(\lambda,\mu)\in \Fer(2,n)$, and that we have in this case

  \begin{equation} \label{rapp di Bn} \rho_{\lambda, \mu}\simeq \Ind_{B_s\times B_{n-s}}^{B_n} \left(\tilde \rho_\lambda \odot (\gamma_{n-s}
  \otimes \tilde \rho_\mu) \right),
\end{equation}
where $s=|\lambda|$.

For $S\subseteq [2m]$ let 
$$\Delta_S\eqdef \{g\in I(2,2m):\,\fix_0(g)=\fix_1(g)=0\textrm{ and } \{i\in[n]:z_i(g)=0\}=S\},$$
and
$$C_S=\sum_{v\in \Delta_S}C_v\in M.$$
\begin{lem} \label{1-rowed rapp}
For all $p_0,p_1\in \mathbb N$ such that $p_0+p_1=m$, the subspace of $M_m$ spanned by all $C_S$ with $|S|=2p_{0}$,
is an irreducible submodule of $(M_m,\varphi)$ affording the representation  $\rho_{\iota_{2p_{0}},\iota_{2p_{1}}}$.
\end{lem}
\begin{proof}
Let us consider the
 $1$-dimensional subspace $\mathbb C C_{[2p_{0}]}$ of $M_m$.

Let us identify the subgroup $B_{2p_{0}}\times B_{2p_{1}}$ of
$B_{2m}$ with the group of the elements permuting "separately" the
first $2p_{0}$ integers and the remaining $2p_{1}$ integers:
$$B_{2p_{0}}\times B_{2p_{1}}\simeq \{g \in B_{2m} :|g|(i)\in [2p_{0}]  \,\forall \,i\in [2p_{0}]\},$$
and we let $\psi=\varphi|_{B_{2p_{0}}\times B_{2p_{1}}}$. We have
\begin{eqnarray*}
\psi(g_1,g_2)(C_{[2p_{0}]})&=&\psi(g_1,g_2)\big(\sum_{v
\in\Delta_{[2p_{0}]}}C_v\big)=\sum_{v\in\Delta_{[2p_{0}]}}\psi(g_1,g_2) (C_v)\\
&=&\sum_{v \in\Delta_{[2p_{0}]}}(-1)^{<g_2,v>}|g_1g_2|v|g_1g_2|^{-1}=\sum_{v
\in\Delta_{[2p_{0}]}}(-1)^{z(g_2)}|g_1g_2|v|g_1g_2|^{-1}\\
&=&(-1)^{z(g_2)}\sum_{v
\in\Delta_{[2p_{0}]}}|g_1g_2|v|g_1g_2|^{-1}=(-1)^{z(g_2)}C_{[2p_{0}]},
\end{eqnarray*}
since, clearly, the map $v\mapsto |g_1g_2|v|g_1g_2|^{-1}$ is a permutation of $\Delta_{[2p_{0}]}$.
Therefore, we have that $(\mathbb C C_{[2p_{0}]},\psi) $ is a representation of $B_{2p_{0}}\times B_{2p_{1}}$ and that it is isomorphic to the representation $\tilde \rho_{\iota_{2p_{0}}}\odot (\gamma_{2p_{1}} \otimes
\tilde \rho_{\iota_{2p_{1}}})$. By the description of the irreducible representations of $B_n$ given in  \eqref{rapp di Bn} we have that 
$$
\Ind_{B_{2p_{0}}\times B_{2p_{1}}}^{B_{2m}}(\mathbb C C_{[2p_{0}]},\psi)\cong \rho_{\iota_{2p_{0}},\iota_{2p_{1}}}.
$$
Now we can observe that, by construction,
$B_{2p_{0}}\times B_{2p_{1}}$ is the stabilizer in $B_{2m}$ of $v$ with respect to the absolute conjugation and that
$$\{C_S:|S|=2p_{0}\}=\{C\in M_{m}:C=\sum_{v\in \Delta_{[2p_{0}]}}C_{|g|v|g|^{-1}}\textrm{ for some }g\in B_{2m}\}.$$ From these facts we  deduce that we also have
$$\Ind_{B_{2p_{0}}\times
B_{2p_{1}}}^{B_{2m}}(\mathbb{C}\,C_{{[2p_{0}]}},\psi)=\bigoplus_{S\subseteq
[2m], |S|=2p_{0}} \mathbb{C}\,C_{S},$$
and the proof is complete.
\end{proof}
\begin{prop}\label{b1)}
For all $m>0$, we have
$$
(M_m,\varphi)\downarrow_{B_{2m-1}}\cong(M_{m-1},\varphi)\uparrow^{B_{2m-1}}.
$$
\end{prop}
\begin{proof}

 For brevity, for all $p_{0},p_{1}\in \mathbb N$ such that $p_{0}+p_{1}=m$, we denote the $B_{2m}$-module $M(c_{0,0,p_{0},p_{1}})$ with $M_{p_{0},p_{1}}$.
 Via the representation $\varphi$, the vector space
$M_m$ naturally splits as a $B_{2m}$-module as it does via $\varrho$:
$$M_m=\bigoplus_{p_0+p_1=m} M_{p_{0},p_{1}}.$$ 

We consider the action of $B_{2m-1}$ on each class $c_{0,0,
p_{0},p_{1}}$ and it is clear that $z_{2m}(v)=z_{2m}(|g|v|g|^{-1})$ for all $v\in B_{2m}$ and $g\in B_{2m-1}$.
In particular, each $M_{p_{0},p_{1}}$ splits, as a $B_{2m-1}$-module, into two submodules according to the color of
$2m$. More precisely, if we denote by 
\begin{eqnarray*}M_{p_{0},p_{1}}^0&\eqdef & \Span\{C_v:v\in c_{0,0,p_{0},p_{1}}\textrm{ and } z_{2m}(v)=0\};\\
 M_{p_{0},p_{1}}^1&\eqdef & \Span\{C_v:v\in c_{0,0,p_{0},p_{1}}\textrm{ and } z_{2m}(v)=1\},
\end{eqnarray*}
we have
$$
M_{p_{0},p_{1}}=M_{p_{0},p_{1}}^0\oplus M_{p_{0},p_{1}}^1
$$
as $B_{2m-1}$-modules, and hence we also have the following decomposition of $M_m$ as a $B_{2m-1}$-module
$$M_m\downarrow_{B_{2m-1}}=\bigoplus_{p_0+p_1=m} \left(M_{p_{0},p_{1}}^0\mbox{ $ \bigoplus $ }M_{p_{0},p_{1}}^1\right).$$

Let us consider the involutions $v_{p_{0},p_{1}}^{0}$, with $p_{0}\neq 0$, and $v_{p_{0},p_{1}}^{1}$, with $p_{1}\neq 0$,  given by
\begin{align*}
v_{p_{0},p_{1}}^{0}\eqdef&[(2,1,4,3,...,2m,2m-1);\underbrace{0,0,...0}_{2(p_{0}-1)},\underbrace{1,...,1}_{2p_{1}},0,0)];\\
v_{p_{0},p_{1}}^{1}\eqdef&[(2,1,4,3,...,2m,2m-1);\underbrace{0,0,...0}_{2p_{0}},\underbrace{1,...,1}_{2p_{1}}].
\end{align*}
We observe that $M_{p_{0},p_{1}}^0$ and $M_{p_{0},p_{1}}^1$ are spanned by all the elements $C_v$ as $v$ varies in the $S_{2m-1}$-conjugacy classes of $v_{p_{0},p_{1}}^{0}$ and $v_{p_{0},p_{1}}^{1}$ respectively,
and so we can express them as induced representations of linear representations of the stabilizers of these elements with respect to the absolute conjugation in $B_{2m-1}$. Namely, if we let
$$H_{p_{0},p_{1}}^{0}\eqdef\{g \in B_{2m-1}:
|g|v_{p_{0},p_{1}}^{0}|g|^{-1}=v_{p_{0},p_{1}}^{0}\},$$ $$H_{p_{0},p_{1}}^{1}\eqdef\{g \in
B_{2m-1}: |g|v_{p_{0},p_{1}}^{1}|g|^{-1}=v_{p_{0},p_{1}}^{1}\},$$ we have then

$$(M_{p_{0},p_{1}}^0, \varphi)\cong \Ind_{H_{p_{0},p_{1}}^{0}}^{B_{2m-1}}( \pi_{p_{0},p_{1}}^{0})\,\,\textrm{ and }\,\,
(M_{p_{0},p_{1}}^1,\varphi)\cong\Ind_{H_{p_{0},p_{1}}^{1}}^{B_{2m-1}}(\pi_{p_{0},p_{1}}^{1}),$$
where
$$\begin{array}{rccl}
\pi_{p_{0},p_{1}}^{0}:&H_{p_{0},p_{1}}^{0}&\rightarrow &\mathbb C^*\\
&g&\mapsto &(-1)^{<g,v_{p_{0},p_{1}}^{0}>}
\end{array}
\,\,\textrm{ and }\,\,
\begin{array}{rccl}
\pi_{p_{0},p_{1}}^{1}:&H_{p_{0},p_{1}}^{1}&\rightarrow &\mathbb C^*\\
&g&\mapsto&(-1)^{<g,v_{p_{0},p_{1}}^{1}>}.
\end{array}
$$

Let us now turn to $M_{m-1}$: arguing as in $M_m$, we have
$$M_{m-1}=\bigoplus_{q_{0}+q_{1}=m-1} M_{q_{0},q_{1}}.$$
As above, $M_{q_{0},q_{1}}$ can be written by means of an induction
from the stabilizer of an involution in $c_{0,0,q_{0},q_{1}}$ with respect to the absolute conjugation. For every
$q_{0},q_{1}$ such that $q_{0}+q_{1}=m-1$, let us consider the vector $u_{q_{0},q_{1}}$ given
by
\begin{equation*}
u_{q_{0},q_{1}}\eqdef[(2,1,4,3,\dots,2m-2,2m-3);\underbrace{0,0,\dots,0}_{2q_{0}},\underbrace{1,\hdots,1}_{2q_1}]
\end{equation*}
 and let
$$K_{q_{0},q_{1}}\eqdef\{g \in B_{2m-2}: |g|u_{q_{0},q_{1}}|g|^{-1}=u_{q_{0},q_{1}}\}.$$ Then
$$(M_{q_{0},q_{1}}, \varphi)=\Ind_{K_{q_{0},q_{1}}}^{B_{2m-2}}(\pi_{q_{0},q_{1}}),$$
where
$$\begin{array}{rccl}
\pi_{q_{0},q_{1}}:&K_{q_{0},q_{1}}&\rightarrow &\mathbb{C}^*\\
&g&\mapsto &(-1)^{<g,u_{q_{0},q_{1}}>}.
\end{array}
$$
Summing up, observing that $M_{0,m}^{0} =M_{m,0}^{1} =\{0\}$, we have
\begin{eqnarray*}M_m\downarrow_{B_{2m-1}}&=&\bigoplus_{p_0+p_1=m}(M_{p_{0},p_{1}}^{0}\oplus M_{p_{0},p_{1}}^{1})= \bigoplus_{q_0+q_1=m-1}(M_{q_{0}+1,q_{1}}^{0}\oplus M_{q_{0},q_{1}+1}^{1})\\
&\cong& \bigoplus_{q_0+q_1=m-1}\left(\Ind_{H_{q_0+1,q_1}^0}^{B_{2m-1}}(\pi_{q_0+1,q_1}^0) \mbox{ $ \bigoplus $ } \Ind_{H_{q_{0},q_{1}+1}^{1}}^{B_{2m-1}}(\pi_{q_{0},q_1+1}^1) \right)
\end{eqnarray*}
 and
$$M_{m-1}\uparrow^{B_{2m-1}}\cong\Ind_{B_{2m-2}}^{B_{2m-1}}\Big(\bigoplus_{q_{0}+q_{1}=m-1} \Ind_{K_{q_{0},q_{1}}}^{B_{2m-2}}
({\pi_{q_{0},q_{1}}})\Big).$$ So, to prove the statement it
is enough to show that
\begin{align*}\bigoplus_{q_{0}+q_{1}=m-1} \left( \Ind_{H_{q_0+1,q_1 }^0}^{B_{2m-1}}(\pi_{q_0+1,q_1}^0)\mbox{ $ \bigoplus $ }
\Ind_{H_{q_{0},q_{1}+1}^{1}}^{B_{2m-1}}(\pi_{q_0,q_1+1}^1)\right) \\
\cong \Ind_{B_{2m-2}}^{B_{2m-1}}\Big(\bigoplus_{q_{0}+q_{1}=m-1}
\Ind_{K_{q_{0},q_{1}}}^{B_{2m-2}}(\pi_{q_{0},q_{1}})\Big).
\end{align*}
As the induction commutes with the direct sum and has the
transitivity property, the last equality is equivalent to
\begin{equation}\label{eccola}
 \bigoplus_{q_{0}+q_{1}=m-1} \big( \Ind_{H_{q_{0}+1,q_{1}}^{0}}^{B_{2m-1}}(\pi_{q_0+1,q_1}^0)\mbox{ $ \bigoplus $ }
 \Ind_{H_{q_{0},q_{1}+1}^{1}}^{B_{2m-1}}(\pi_{q_0,q_1+1}^1)\big) \cong \bigoplus_{q_{0}+q_{1}=m-1}\Ind_{K_{q_{0},q_{1}}}^{B_{2m-1}}
 (\pi_{q_{0},q_{1}}).
\end{equation}
The choice of the vectors $v_{p_{0},p_{1}}^{0}$, $v_{p_{0},p_{1}}^{1}$ and
$u_{q_{0},q_{1}}$ leads to:
\begin{align*}
H_{p_{0},p_{1}}^{0}&=\{g \in B_{2m-1}:|g| \in S_{2(p_{0}-1)}\times S_{2p_{1}}, |g|(i+1)=|g|(i)\pm 1\, \forall \,i
\mbox{ odd, } 0< i< 2m\}; \\
H_{p_{0},p_{1}}^{1}&= \{g \in B_{2m-1}:|g| \in S_{2p_{0}}\times S_{2(p_{1}-1)}, |g|(i+1)=|g|(i)\pm 1\, \forall \,i
\mbox{ odd, } 0< i< 2m\}; \\
K_{q_{0},q_{1}} &=\{g \in B_{2m-2}:|g| \in S_{2q_{0}}\times
S_{2(q_{1}-1)}, |g|(i+1)=|g|(i)\pm 1\, \forall \,i \mbox{ odd, }
0< i< 2m-2\}
\end{align*}
where, as usual $S_h\times S_k=\{\sigma\in S_{h+k}:\sigma(i)\leq h \textrm{ for all }i\leq h\}$.
We therefore make the crucial observation that $$H_{q_{0}+1,q_{1}}^{0}=H_{q_{0},q_{1}+1}^{1},$$
so that to prove \eqref{eccola} it is enough to show that
\begin{equation}\label{eccolab} \Ind_{H_{q_{0},q_{1}+1}^{1}}^{B_{2m-1}}\left(\pi_{q_{0}+1,q_{1}}^0\mbox{$\bigoplus $}\, \pi_{q_0,q_1+1}^1\right)\cong \Ind_{K_{q_{0},q_{1}}}^{B_{2m-1}}
 (\pi_{q_{0},q_{1}}).
\end{equation}

Now we also observe that $K_{q_{0},q_{1}}$ is a subgroup of $H_{q_{0},q_{1}+1}^{1}$ (of index 2), so that the right-hand
side of $\eqref{eccolab}$ becomes  $\Ind_{H_{q_{0},q_{1}+1}^{1}}^{B_{2m-1}} \left
(\Ind_{K_{q_{0},q_{1}}} ^{H_{q_{0},q_{1}+1}^{1}} (\pi_{q_{0},q_{1}})\right )$ and therefore we are left to prove that

\begin{equation}\label{eccolac}
 \pi_{q_{0}+1,q_{1}}^0\mbox{$\bigoplus $}\, \pi_{q_0,q_1+1}^1=\Ind_{K_{q_{0},q_{1}}} ^{H_{q_{0},q_{1}+1}^{1}} (\pi_{q_{0}.q_{1}}).
\end{equation}

 If we let $\chi_1$ be the character of $\pi_{q_{0}+1,q_{1}}^0\mbox{$\bigoplus $}\, \pi_{q_0,q_1+1}^1$ and $\chi_2$ be the character of $\Ind_{K_{q_{0},q_{1}}} ^{H_{q_{0},q_{1}+1}^{1}} (\pi_{q_{0},q_{1}})$ we only have to show that $\chi_1(g)=\chi_2(g)$ for all $g\in H_{q_{0},q_{1}+1}^{1}$.

We have
\begin{eqnarray*}
\chi_1(g)&=&(-1)^{<g,v_{q_{0}+1,q_{1}}^0>}+ (-1)^{<g,v_{q_{0},q_{1}+1}^{1}>}\\
 &=&(-1)^{\sum_{i=2q_{0}+1}^{2m-2}z_i(g)}+(-1)^{\sum_{i=2q_{0}+1}^{2m}z_i(g)}\\
 &=&(1+(-1)^{z_{2m-1}(g)})(-1)^{\sum_{i=2q_{0}+1}^{2m-2}z_i(g)},
\end{eqnarray*}
where we have used the fact that $z_{2m}(g)=0$, since $g\in B_{2m-1}$.

As for the character $\chi_2$, we observe that $K_{q_{0},q_{1}}$ is
the subgroup of $H_{q_{0},q_{1}+1}^{1}$ of all the elements $g$ with
 $z_{2m-1}(g)=0$. So we may take 

$$C=\{\Id_{B_{2m-1}},\,\sigma\eqdef[1,2,\ldots,2m-2,-(2m-1),2m]\},$$
as a system of coset representatives of $H_{q_{0},q_{1}+1}^{1}/K_{q_{0},q_{1}}$.
Therefore the induced character $\chi_2$ is given by
$$\chi_2(g)=\sum_{\begin{subarray}{c}
 h \in C \\
 h^{-1}g h \in K_{p_{0},p_{1}}
 \end{subarray}} \chi_{\pi_{q_0,q_1}}(h^{-1}g h).$$ Since $g(2m-1)=\pm(2m-1)$ we have that $g \notin K_{q_{0},q_{1}} \Leftrightarrow \,
\forall \, h \in C, h^{-1}g h \notin K_{q_{0},q_{1}}$, and hence
$$\chi_2(g)=0 \,\, \forall\, g \in H_{q_{0},q_{1}+1}^{1}| z_{2m-1}(g)=1,$$
which agrees with $\chi_1(g)$.\\
So we are left to compute $\chi_2(g)$, where  $g$ satisfies
$z_{2m-1}(g)=0$. In this case we have $g(2m-1)=2m-1$ which implies $\sigma^{-1} g \sigma=g$, and hence
\begin{eqnarray*}\chi_2(g)&=&(-1)^{<g,u_{q_{0},q_{1}}>}+(-1)^{<\sigma^{-1} g\sigma,u_{q_{0},q_{1}}>}\\
&=&2(-1)^{<g,u_{q_{0},q_{1}}>}\\
&=& 2(-1)^{\sum_{i=2q_{0}+1}^{2m-2}z_i(g)}.
\end{eqnarray*}

We conclude that $\chi_1(g)=\chi_2(g)$ for all $g\in H_{q_{0},q_{1}+1}^1$, so $\eqref{eccolac}$ is satisfied and the proof is complete.
\end{proof}

\begin{thm} \label{nofixedpoints}
 For all $m\in \mathbb N$, $(M_m, \varphi)$ is a $B_{2m}$-module isomorphic to the direct sum of all
the irreducible representations of $B_{2m}$ indexed by the even diagrams of $\Fer(2,2m)$,
each of such representations occurring once. 
    \end{thm}
\begin{proof}
It is enough to check that the representations $(M_m, \varphi)$ satisfy the conditions b0), b1), b2) of Proposition \ref{indres}.

Condition b0) is trivial.

In order to check condition b1), we
have to find two submodules of $M_m$ which are isomorphic to the
representations indexed by $(\iota_{2m}, \emptyset)$ and
$(\emptyset, \iota_{2m})$. By Lemma \ref{1-rowed rapp}, they correspond
respectively to $$\rho_{\iota_{2m}, \emptyset}=
(\mathbb{C}\,C_{[2m]}, \varphi) \quad \mbox{ and }
\rho_{\emptyset, \iota_{2m}}=(\mathbb{C}\,C_{\emptyset},
\varphi).$$

Condition b2) is the content of Proposition \ref{b1)} and the proof is complete.
\end{proof}
We are now in a position to fully describe the irreducible decomposition of the submodules $M_{p_{0},p_{1}}$ of $M_m$ via the representation $\varphi$. 
\begin{thm} \label{nofixed points refined}We have
$$(M_{p_{0},p_{1}}, \varphi)\cong\bigoplus_{\substack{
                                                            |\lambda|=2p_{0}, |\mu|=2p_{1}\\
                                                            \lambda,
                                                            \mu \textrm{ with
                                                            no odd
                                                            rows}
                                                            }}
 \rho_{\lambda,\mu}.$$

\end{thm}
\begin{proof}
We start by showing that there exist representations $\sigma$  of $S_{2p_{0}}$ and  $\tau$ of  $S_{2p_{1}}$ such that 
\begin{equation}\label{sita}
(M_{p_{0},p_{1}},\varphi )\cong\Ind_{B_{2p_{0}}\times B_{2p_{1}}}^{B_{2m}}(\tilde \sigma \odot (\gamma_{2p_{1}} \otimes \tilde \tau)),\end{equation}
where $\tilde \sigma$ and $\tilde \tau$ are the natural extensions of  $\sigma$ and $\tau $ to $B_{2p_{0}}$ and to $B_{2p_{1}}$, respectively.

Recall the definition of $\Delta_S$ given before the statement of Lemma \ref{1-rowed rapp}. If we let $M_{S}\eqdef \Span\{C_v:v\in \Delta_{S}\}$, it is clear that 
$$M_{p_{0},p_{1}}=M_{[2p_{0}]}\big \uparrow_{B_{2p_{0}}\times B_{2p_{1}}}^{B_{2m}}.$$
Now, since 
\begin{eqnarray*}\Delta_{[2p_{0}]}&=&\{v\in B_{2m}: v\textrm{ is an involution in }S_{2p_{0}}\times -(S_{2p_{1}})\}\\
&=&\{v:v=v'v''\textrm{with $v'$ involution in $S_{2p_{0}}$ and $-v''$ involution in $S_{2p_{1}}$}\} ,
\end{eqnarray*}
we deduce the isomorphism of vector spaces $M_{[2p_{0}]}\cong M'\otimes M''$, where $$M'=\Span\{C_{v'}:v'\textrm{ is an involution in }S_{2p_0}\}$$ and $$M''=\Span\{C_{v''}:v''\textrm{ is an involution in }S_{2p_1}\},$$ the isomorphism being given by $C_{v'v''}\leftrightarrow C_{v'}\otimes C_{-v''}$.
If $g=(g',g'')\in B_{2p_{0}}\times B_{2p_{1}}$ and $v=v'(-v'')\in \Delta_{[2p_{0}]}$ we have
\begin{eqnarray*}
\varphi(g)C_{v'}\otimes C_{v''}&\leftrightarrow& \varphi(g)C_v\\
&=& (-1)^{<g,v>}C_{|g|v|g|^{-1}}\\
&=& (-1)^{<g'',-v''>}C_{|g'|v'|g'|^{-1}|g''|(-v'')|g''|^{-1}}\\
&\leftrightarrow& C_{|g'|v'|g'|^{-1}}\otimes (-1)^{z(g_2)}C_{|g_2|v''|g_2|^{-1}}.
\end{eqnarray*}
and Equation \eqref{sita} follows. Now the full result is a direct consequence of the irreducible decomposition of the representations $\sigma$ and $\tau$, the description of the irreducible representations given in \eqref{rapp di Bn}, and Theorem \ref{nofixedpoints}.
\end{proof}

The next goal is to describe the relationship between the irreducible decomposition of the representations $\varphi$ and $\varrho$.

Recall that $\varrho(g)(C_v)=(-1)^{\inv_v(g)}\varphi(g)(C_v)$; we will show that the factor $(-1)^{\inv_v(g)}$ simply
exchanges the roles of rows and columns of the Ferrers diagrams appearing in the irreducible decomposition of the 
$B_{2m}$-modules $(M_m, \varphi)$ and $(M_m,\varrho)$. 

\begin{lem}\label{invinv} For $p_0,p_1\in \mathbb N$ with $p_0+p_1=m$ let $u_{p_{0},p_{1}}$ and $K_{p_{0},p_{1}}$ be (as in Proposition \ref{b1)}):
\begin{align*}
u_{p_{0},p_{1}}&=[(2,1,4,3,...,2m,2m-1);\underbrace{0,0,...0}_{2p_{0}},\underbrace{1,...,1}_{2p_{1}}];\\
K_{p_{0},p_{1}}&=\{g \in B_{2m}:|g| \in S_{2p_{0}}\times
S_{2p_{1}}, |g|(i+1)=|g|(i)\pm 1\, \forall \,i \mbox{ odd, }
0< i< 2m\}.
\end{align*}
Then, for every $g\in K_{p_{0},p_{1}}$, we have
$$\inv_{u_{p_{0},p_{1}}}(g)\equiv \inv(|g|) \mod 2.$$
\end{lem}
\begin{proof}
We can clearly assume that $g=|g|$. Let $\{i,j\}$ be in   $\Inv(g)$, but not in $\Pair(|u_{p_{0},p_{1}}|)$. As $u_{p_{0},p_{1}}$
is an involution satisfying $\fix_0(u_{p_{0},p_{1}})=\fix_1(u_{p_{0},p_{1}})=0$, there exist unique $h$ and $k$ such that $\{i,h\}$ and $\{j,k\}$
belong to $\Pair(|u_{p_{0},p_{1}}|)$. We will show that $\{h,k\}$ - which does not belong to
 $\Pair(|u_{p_{0},p_{1}}|)$ - is an element of $\Inv(g)$. In this way, every pair $\{i,j\} \in \Inv(g)\setminus \Pair(|u_{p_{0},p_{1}}|)$
 can be associated to exactly another, so $|\Inv(g)\setminus \Pair(|u_{p_{0},p_{1}|})|$
 is even and we get the result.

We can assume that $i<j$ (hence $g(i)>g(j)$) throughout.
Observe that we know from the form of $u_{p_{0},p_{1}}$ that $i=h\pm 1$, and $j=k\pm 1$, depending on the parity of $i$ and $j$.
Nevertheless, in all cases, we always obtain $h<k$ (since the four integers $i,j,h,k$ are distinct), so that the claim to prove is always $g(h)>g(k)$. But the definition of $K_{p_{0},p_{1}}$ ensures that $g(h)=g(i)\pm1$ and $g(k)=g(j)\pm1$. The result follows since  
$g(i)>g(j)$, and the fact that the four integers $g(i),g(j), g(h),g(k)$ are distinct.
\end{proof}
We recall the following general result in representation theory.
Let $G$ be a finite group, $H<G$. Let $\vartheta$, $\tau$ be
representations respectively of $G$ and of $H$. We have
\begin{equation}\label{theta}
(\vartheta\downarrow_H \otimes \,\,\tau)\uparrow^G\cong
\vartheta\otimes \,\,(\tau\uparrow^G).
\end{equation}
Let us denote by $\sigma_n$ the linear representation of $B_{n}$ given by $\sigma_n(g)=(-1)^{\inv(|g|)}$.  
\begin{lem}\label{altconj}
 For all $(\lambda,\mu)\in \Fer(2,n)$ we have

$$\sigma_{n}(g)\otimes\,\,\rho_{\lambda, \mu}=\rho_{\lambda',\mu'},$$
where $\lambda'$ and $\mu'$  denote the conjugate partitions of $\lambda$ and $\mu$ respectively.
\end{lem}
\begin{proof}
 We recall the following well-known analogous fact for the symmetric group. We have
\begin{equation}\label{conj}
 \epsilon \otimes \rho_{\lambda}=\rho_{\lambda'},
\end{equation} 
where $\epsilon(g)\eqdef(-1)^{\inv(g)}$ denotes the alternating representation. If we let $k=|\lambda|$ then, by Equations \eqref{theta} and  \eqref{conj}, we have
\begin{align*}
\sigma_{n}\otimes \,\,\,\rho_{\lambda, \mu}&=\sigma_{n}\otimes \,\Ind_{B_{k}\times B_{n-k}}^{B_{n}}( \tilde \rho_{\lambda} \odot (\gamma_{n-k} \otimes \, \tilde \rho_{\mu}))\\
&\cong \Ind_{B_{k}\times B_{n-k}}^{B_{n}}\big(\sigma_{n}\downarrow_{B_{k}\times B_{n-k}}\otimes \, \,(\tilde \rho_\lambda \odot (\gamma_{n-k} \otimes \, \tilde \rho_\mu))\big)\\
 &=\Ind_{B_{k}\times B_{n-k}}^{B_{n}}\big((\sigma_{n}\downarrow_{B_{k}}\otimes \, \tilde \rho_\lambda)\odot (\sigma_{n}\downarrow_{B_{n-k}}\otimes \, \gamma_{n-k}\otimes \, \tilde \rho_\mu) \big)\\
&=\Ind_{B_{k}\times B_{n-k}}^{B_{n}}\big(\widetilde{(\epsilon \otimes  \rho_\lambda)}\odot (\gamma_{n-k}\otimes \widetilde{(\epsilon \otimes \, \rho_\mu)}) \big)\\
&=\Ind_{B_{k}\times B_{n-k}}^{B_{n}}\big(\tilde \rho_{\lambda'}\odot (\gamma_{n-k}\otimes \, \tilde\rho_{\mu'}) \big)\\
&=\rho_{\lambda',\mu'},
\end{align*}
 and the proof is complete.
\end{proof}

\begin{thm} \label{nofixed points refined 2}We have
$$(M_{p_{0},p_{1}},  \varrho)\cong\bigoplus_{\substack{
                                                            |\lambda|=2p_{0}, |\mu|=2p_{1}\\
                                                            \lambda,
                                                            \mu \textrm{ with
                                                            no odd
                                                            columns}
                                                            }}
 \rho_{\lambda,\mu}.$$

\end{thm}
\begin{proof} Let us consider the linear representation of $K_{p_{0},p_{1}}$
$$(-1)^{\inv_{u_{p_{0},p_{1}}}(g)}\pi_{p_{0},p_{1}}(g)=(-1)^{\inv_{u_{p_{0},p_{1}}}(g)}(-1)^{<g,u_{p_{0},p_{1}}>}.$$

We have
\begin{align*}
(M_{p_{0},p_{1}},\varrho)&=((-1)^{\inv_{u_{p_{0},p_{1}}}(g)}\pi_{p_{0},p_{1}})\big \uparrow_{K_{p_{0},p_{1}}}^{B_{2m}}
=((-1)^{\inv(|g|)}\pi_{p_{0},p_{1}})\big \uparrow_{K_{p_{0},p_{1}}}^{B_{2m}}\\
&=\left((-1)^{\inv(|g|)}\big \downarrow_{K_{p_{0},p_{1}}}\otimes\,\,\pi_{p_{0},p_{1}})\right)\big \uparrow_{K_{p_{0},p_{1}}}^{B_{2m}}\\
& \cong(-1)^{\inv(|g|)}\otimes\,\,(\pi_{p_{0},p_{1}}\uparrow^{B_{2m}})=
(-1)^{\inv(|g|)}\otimes\,\,(M_{p_{0},p_{1}},\varphi),
\end{align*}
where we have used Lemma \ref{invinv} in the first line and Equation \eqref{theta} in the last line of the previous equalities.
Now the result follows from Lemma \ref{altconj} and Theorem \ref{nofixed points refined}.
\end{proof}

\section {$B_n$: the proof of the full result}\label{full}

In this section we will give a complete proof in the case of $B_n$ of Theorem \ref{mainth} that, by Proposition \ref{Sn-classes}, can be restated in the following slightly different but equivalent form.
\begin{thm}\label{mainB}
For all $f_{0}, f_{1}, p_{0},p_{1}\in \mathbb N$ such that $f_{0} +f_{1}+2p_{0}+2p_{1}=n$ we have
$$
(M(c_{f_{0}, f_{1}, p_{0},p_{1}}),\varrho)\cong \bigoplus_{\substack{|\lambda|=2p_{0}+f_{0}, |\mu|=2p_{1}+f_{1}\\ \lambda  \textrm{ with exactly $f_{0}$ odd columns}\\  \mu  \textrm{ with exactly $f_{1}$ odd columns}}} \rho_{\lambda,\mu}.
$$

\end{thm}
\begin{proof}
Let $m=p_{0}+p_{1}$ and consider the space $M(c_{0,0,p_{0},p_{1}})$: it is a $B_{2m}$-module via
the representation
$$\Pi_{p_{0},p_{1}}\eqdef(M(c_{0,0,p_{0},p_{1}}), \varrho)=\Ind_{K_{p_{0},p_{1}}}^{B_{2m}}(\tau_{p_{0},p_{1}}),$$
where $\tau_{p_{0},p_{1}}$ is the linear $K_{p_{0},p_{1}}$ representation given by
$\tau_{p_{0},p_{1}}(g)=(-1)^{\inv(|g|)}\pi_{p_{0},p_{1}}(g)$.
From Theorem \ref{nofixed points refined 2}, we know that it is the
multiplicity-free direct sum of all representations indexed by
pairs of diagrams $(\lambda,\mu)$ where $\lambda$ and $\mu$ have
even columns only, and $|\lambda|= 2p_{0}, |\mu|=2p_{1}$. \\
We will first show that
\begin{equation} \label{pappappero}
(M(c_{f_{0}, f_{1}, p_{0}, p_{1}}),\varrho)= \Ind_{B_{2m}\times
B_{n-2m}}^{B^n}(\Pi_{p_{0},p_{1}}\odot \rho_{\iota_{f_{0}},\iota_{f_{1}}}).
\end{equation}
Let us argue with the same strategy as in \S\ref{parz}. We define the involution $u$ representing the $S_n$-conjugacy class $c_{f_{0}, f_{1}, p_{0}, p_{1}}$ as follows:
$$u=[(2,1,4,3,\ldots,2m, 2m-1, 2m+1, \ldots, n); \underbrace{0, \ldots, 0,}_{2p_{0}}
 \underbrace{1, \ldots, 1}_{2p_{1}}, \underbrace{0, \ldots, 0}_{f_{0}},\underbrace{ 1, \ldots, 1}_{f_{1}}].$$
We have that the stabilizer of $u$ with respect to the absolute conjugation is $\{g\in B_n :\,|g|u|g|^{-1}=u\}=K_{p_{0},p_{1}} \times B_{f_{0}} \times B_{f_{1}}$, and we can easily check that
$$(M(c_{f_{0}, f_{1}, p_{0}, p_{1}}), \varrho )=
\Ind_{K_{p_{0},p_{1}} \times B_{f_{0}} \times B_{f_{1}}}^{B_n} \big(
\tau_{p_{0},p_{1}}\odot \rho_{\iota_{f_{0}},\emptyset}\odot \rho_{\emptyset, \iota_{f_{1}}}\big).$$
We recall the following identity of induced representations: if $H<G$ and $H'<G'$ we have
\begin{equation}\label{inducedproducts}
\Ind_{H\times H'}^{G\times G'}(\rho\odot \rho')=\Ind_H^G(\rho)\odot \Ind_{H'}^{G'}(\rho'), 
\end{equation}
where $\rho$ is a representation of $H$ and $\rho'$ a representation of $H'$.
 So we have
\begin{eqnarray*}
 (M(c_{f_{0}, f_{1}, p_{0}, p_{1}}), \varrho )&=&\Ind_{K_{p_{0},p_{1}} \times B_{f_{0}} \times B_{f_{1}}}^{B_n} \big(
\tau_{p_{0},p_{1}}\odot \rho_{\iota_{f_{0}},\emptyset}\odot \rho_{\emptyset, \iota_{f_{1}}}\big)\\
&=&\Ind_{B_{2m}\times B_{n-2m}}^{B^n}\big( \Ind_{K_{p_{0},p_{1}} \times B_{f_{0}} \times
B_{f_{1}}}^{B_{2m}\times B_{n-2m}} (\tau_{p_{0},p_{1}}\odot \rho_{\iota_{f_{0}},\emptyset}\odot \rho_{\emptyset, \iota_{f_{1}}}\big)\\
&=&\Ind_{B_{2m}\times B_{n-2m}}^{B^n}\big( \Ind_{K_{p_{0},p_{1}}}^{B_{2m}}(\tau_{p_{0},p_{1}})\odot \Ind_{B_{f} \times B_{f_{1}}}^{B_{n-2m}}(\rho_{\iota_{f_{0}},\emptyset}\odot \rho_{\emptyset, \iota_{f_{1}}})\big)\\
&=&\Ind_{B_{2m}\times B_{n-2m}}^{B^n}(\Pi_{p_{0},p_{1}}\odot \rho_{\iota_{f_{0}},\iota_{f_{1}}})
\end{eqnarray*} and Equation
 \eqref{pappappero} is achieved.
Now the result follows from Theorem \ref{nofixed points refined 2} and the following result which is the analogue in type $B$ of the well-known Pieri rule (see \cite[Lemma 6.1.3]{GP}).
\end{proof}
\begin{prop}
Let $\rho_{\lambda,\mu}$ be any irreducible representation of $B_m$. Then
$$
\Ind_{B_m\times B_n}^{B_{n+m}}(\rho_{\lambda,\mu}\odot \rho_{\iota_f,\iota_{n-f}})=\bigoplus \rho_{\nu, \xi},
$$
where the direct sum runs through all $(\nu,\xi)\in \Fer(2,n+m)$ such that $\nu$ is obtained from $\lambda$ by adding $f$ boxes to its Ferrers diagram, no two in the same column, and $\xi$ is obtained from $\mu$ by adding $n-f$ boxes to its Ferrers diagram, no two in the same column.
 \end{prop}

\begin{exa} For every $f_{0},f_{1}\in [0,n]$, with $f_{0}+f_{1}=n$ let us consider the set
$\Sh(c_{f_{0}, f_{1},0,0})$. Since $\iota_k$ is the only $k$-boxed diagram with $k$ odd columns,
$\Sh(c_{f_{0}, f_{1},0,0})$ contains the pair $(\iota_{f_{0}}, \iota_{f_{1}})$ only. Thus we can explicitly find in $(M, \varrho)$
the subspace $V_{\iota_{f_{0}}, \iota_{f_{1}}}$ affording the representation $\rho_{\iota_{f_{0}}, \iota_{f_{1}}}$:
thanks to Theorem \ref{mainth},
\begin{align*}
V_{\iota_{f_{0}}, \iota_{f_{1}}}=M(c_{f_{0},f_{1},0,0})=\Ind_{B_{f_{0}}\times B_{f_{1}}}^{B_n}(\mathbb{C}\, C_{u_{f_{0},f_{1},0,0}}).
\end{align*}
$u_{f_{0},f_{1},0,0}$ being the involution $$u_{f_{0},f_{1},0,0}=[1,2,\ldots,f_{0},-(f_{0}+1),\ldots,-n].$$
In other words 
$$V_{\iota_{f_{0}}, \iota_{f_{1}}}=\Span\{C_v:v\in B_n,\,|v|=\Id,\,\#\{i:z(i)=0\}=f_{0},\,\#\{i:z(i)=1\}=f_{1}\}.$$
\end{exa}

\begin{exa}
Let $v=[-6,4,3,2,-5,-1]=[(6,4,3,2,5,1);1,0,0,0,1,1]\in B_6$.  Then $f_{0}=f_{1}=p_{0}=p_{1}=1$ and the
$S_n$-conjugacy class $c$ of $v$ has 180 elements. Then the
$B_n$-module $M(c)$ is given by the sum of the irreducible
representations indexed by $(\lambda,\mu)\in \Fer(2,n)$ such that
both $\lambda$ and $\mu$ are partitions of 3 and have exactly one
column of odd length. In particular
$$
M(c)\cong \rho_{\left(\begin{picture}(24,9)
\put(0,8){\line(1,0){10}} \put(0,3){\line(1,0){10}}
\put(0,-2){\line(1,0){5}} \put(0,-2){\line(0,1){10}}
\put(5,-2){\line(0,1){10}} \put(10,3){\line(0,1){5}} \put(12,0){,}

\put(17,8){\line(1,0){5}} \put(17,3){\line(1,0){5}}
\put(17,-2){\line(1,0){5}} \put(17,-7){\line(1,0){5}}
\put(17,-7){\line(0,1){15}} \put(22,-7){\line(0,1){15}}
\end{picture}\right)}
\oplus \rho_{\left(\begin{picture}(24,9)
\put(13,8){\line(1,0){10}} \put(13,3){\line(1,0){10}}
\put(13,-2){\line(1,0){5}} \put(13,-2){\line(0,1){10}}
\put(18,-2){\line(0,1){10}} \put(23,3){\line(0,1){5}} \put(8,0){,}

\put(0,8){\line(1,0){5}} \put(0,3){\line(1,0){5}}
\put(0,-2){\line(1,0){5}} \put(0,-7){\line(1,0){5}}
\put(0,-7){\line(0,1){15}} \put(5,-7){\line(0,1){15}}
\end{picture}\right)}
\oplus \rho_{\left(\begin{picture}(29,9) \put(0,8){\line(1,0){10}}
\put(0,3){\line(1,0){10}} \put(0,-2){\line(1,0){5}}
\put(0,-2){\line(0,1){10}} \put(5,-2){\line(0,1){10}}
\put(10,3){\line(0,1){5}} \put(12,0){,}

\put(17,8){\line(1,0){10}} \put(17,3){\line(1,0){10}}
\put(17,-2){\line(1,0){5}} \put(17,-2){\line(0,1){10}}
\put(22,-2){\line(0,1){10}} \put(27,3){\line(0,1){5}}
\end{picture}\right)} \oplus
\rho_{\left(\begin{picture}(19,9) \put(0,8){\line(1,0){5}}
\put(0,3){\line(1,0){5}} \put(0,-2){\line(1,0){5}}
\put(0,-7){\line(1,0){5}} \put(0,-7){\line(0,1){15}}
\put(5,-7){\line(0,1){15}} \put(7,0){,}

\put(12,8){\line(1,0){5}} \put(12,3){\line(1,0){5}}
\put(12,-2){\line(1,0){5}} \put(12,-7){\line(1,0){5}}
\put(12,-7){\line(0,1){15}} \put(17,-7){\line(0,1){15}}
\end{picture}\right)}.
$$
\end{exa}

\section{The general case of wreath products}\label{grn}

In this section we will treat the general case $G=G(r,n)$. To prove Theorem \ref{mainth}, we will be handling the same tools already used in the case of $B_n$. Nevertheless, as some of the results need to be slightly generalized, we will provide an outline of the whole argument in this wider setting.

Let $M$ be the model for $G(r,n)$ described in Theorem \ref{main}. Let $\varphi$ be the representation defined analogously to the case of $B_n$:
\begin{align*}
 \varphi(g): &M \rightarrow M\\
 & \, C_v \mapsto (-1)^{<g,v>}C_{|g|v|g|^{-1}}.
\end{align*} 
The $S_n$-conjugacy classes of absolute involutions of $G(r,n)$ are indexed by $2r$-plets $(f_0,\ldots, f_{r-1},p_0,\ldots,p_{r-1})$ satisfying $f_0+\cdots+f_{r-1}+2(p_0+\cdots+p_{r-1})=n$. These are given by
$$
c_{f_0,\ldots, f_{r-1},p_0,\ldots,p_{r-1}}=\{v\in I(r,n):\fix_{i}(v)=f_i\textrm{ and }\pair_{i}(v)=p_i\, \forall i\in [0,r-1]\}.
$$ 
where
\begin{eqnarray*}
\fix_{i}(v)&=&|\{j\in[n]:v(j)=\zeta_r^ij\}|\\
\pair_i(v)&=&|\{(h,k):1\leq h<k \leq n,\,v(h)=\zeta_r^ik \textrm{ and } v(k)=\zeta_r^i h\}|.
\end{eqnarray*}
The main idea is, again, focusing on the submodule with no fixed points first. 
Our half-way result is
\begin{thm} \label{nofixedpointsG(r,n)}
Let $M_{m,r}$ be the subspace of $M$ spanned by the elements $C_v$ as $v$ varies among all
involutions in $G(r,2m)$ such that
$\fix_0(v)=\fix_1(v)=\ldots=\fix_{r-1}(v)=0$:$$M_{m,r}\eqdef\bigoplus_{p_0+ \cdots+p_{r-1}=m}
M(c_{0,\ldots,0,p_0, \ldots,p_{r-1}}).$$ Then $(M_{m,r}, \varphi)$ is a $G(r,2m)$-module isomorphic to the direct sum of all
the irreducible representations of $G(r,2m)$ indexed by the diagrams of $\Fer(r,2m)$
whose rows have an even number of boxes,
each of such representations occurring once.
    \end{thm}

We state here the $G(r,n)$-generalized version of Proposition \ref{indres}, which will be applied to $M_{m,r}$.
\begin{prop} \label{grnindres}
 Let $\Pi_m^r$ be representations of $G(r,2m)$, $m$ ranging in $\mathbb{N}$.
 Then the following are equivalent:
 \begin{enumerate}
         \item [a)] for every $m$, $\Pi_m^r$ is the direct sum of all the irreducible representations of
         $G(r,2m)$ indexed by $r$-plets of even diagrams, each of such representations occurring once;
         \item [b)] for every $m$,
           \begin{enumerate}
\item[b0)] $\Pi_0^r$ is unidimensional;
 \item[b1)] the module $\Pi_m^r$
               contains the irreducible representations of $G(r,2m)$ indexed by the $r$
               $r$-plets of diagrams $(\emptyset, \ldots, \emptyset, \iota_{2m}, \emptyset, \ldots, \emptyset)$.
               \item [b2)]the following isomorphism holds:
               \begin{equation}\label{grnresind}
                  \Pi_m^r\downarrow_{G(r,2m-1)}\cong \Pi_{m-1}^r\uparrow^{G(r,2m-1)};
               \end{equation}
              
           \end{enumerate}
 \end{enumerate}
\end{prop}
Here is the generalization of the branching rule for $G(r,n)$, which is an essential ingredient for the proof of Proposition \ref{grnindres}. The rest of the proof does not present any other significant change.


\begin{thm}
Let $(\lambda^{(0)},\ldots,\lambda^{(r-1)})\in \Fer(r,n)$.
Then the following holds:
$$\rho_{\lambda^{(0)},\ldots,\lambda^{(r-1)}}\downarrow_{G(r,n-1)}=\bigoplus_{(\mu^{(0)},\ldots,\mu^{(r-1)}) \in \,
R^-_{\lambda^{(0)},\ldots,\lambda^{(r-1)}}}  \rho_{\mu^{(0)},\ldots,\mu^{(r-1)}};$$
$$\rho_{\lambda^{(0)},\ldots,\lambda^{(r-1)}}\uparrow^{G(r,n+1)}=\bigoplus_{(\mu^{(0)},\ldots,\mu^{(r-1)}) \in \,
 R^+_{\lambda^{(0)},\ldots,\lambda^{(r-1)}}} \rho_{\mu^{(0)},\ldots,\mu^{(r-1)}},$$
where we denote by $R^+_{\lambda^{(0)},\ldots,\lambda^{(r-1)}}$ the set of diagrams in $\Fer(r,n+1)$ obtained by adding one box to the diagram $(\lambda^{(0)},\ldots,\lambda^{(r-1)})$, and similarly for $R^-_{\lambda^{(0)},\ldots,\lambda^{(r-1)}}$.
\end{thm}


Let us check that $M_{m,r}$ satisfies properties b) of Proposition \ref{grnindres}, so that Theorem \ref{nofixedpointsG(r,n)} follows.

Property b0) is trivial and so we look for property b1): for $S_0,\ldots,S_{r-1}$  disjoint subsets of $[2m]$ such that $\cup S_i=[2m]$ we let
\begin{align*}
\Delta_{S_0,\ldots,S_{r-1}}\eqdef \{&v\,|\,\, v \mbox { is an absolute involution of } G(r,2m) \mbox{ with: }\\ &\fix_0(v)=\ldots =\fix_{r-1}(v)=0; z_i(v)=j \,\, \mbox{ iff } \,i\,\in S_j\},
\end{align*}
and
$$C_{S_0, \ldots, S_{r-1}}=\sum_{v\in \Delta_{S_0, \ldots, S_{r-1}}}C_v\in M.$$

\begin{lem}\label{1-rowed rapp G(r,n)}
The subspace of $M_{m,r}$ spanned by all $C_{S_0, \ldots, S_{r-1}}$, with $|S_i|=p_i$,
is an irreducible submodule of $(M_{m,r},\varphi)$ affording the representation  $\rho_{\iota_{2p_{0}},\ldots,\iota_{2p_{r-1}}}$.
\end{lem}

\begin{proof}
 This proof can be carried on in the same way as in the case of $B_n$,  relying on Proposition \ref{rapp di grn}.
\end{proof}

Let us turn to property b2). We have to check that $$M_{m,r}\downarrow_{G(r,2m-1)}\cong M_{m-1,r}\uparrow^{G(r,2m-1)}.$$
We let $M_{p_0,\ldots,p_{r-1}}=M(c_{0,\ldots,0,p_0,\ldots,p_{r-1}})$.
First of all, the following decomposition holds:
\begin{align*}
M_{m,r}\downarrow_{G(r,2m-1)}&=\bigoplus_{p_0+\cdots+p_{r-1}=m}
M_{p_0, \ldots,p_{r-1}}\downarrow_{G(r,2m-1)}\\&=\bigoplus_{p_0+\cdots+p_{r-1}=m}\bigoplus_{j=o}^{r-1}M_{p_0,\ldots,p_{r-1}}^j,
\end{align*}
$M_{p_0,\ldots,p_{r-1}}^j$ being the submodule of $M_{p_0, \ldots,p_{r-1}}$ spanned by the absolute involutions $v$ such that $z_{2m}(v)=j$.

As the module $M_{p_0,\ldots,p_{r-1}}^j$ is trivial whenever $p_j=0$, we can reduce ourselves to
$$\bigoplus_{q_0+\cdots+q_{r-1}=m-1}\bigoplus_{j=o}^{r-1}M_{q_0,\ldots,q_j+1, \ldots,q_{r-1}}^j$$
We introduce the absolute involution $$v_{q_0,\ldots,q_{r-1}}^j\eqdef[(2,1,4,3,...,2m,2m-1);
\underbrace{0,0,...0}_{2q_0},\underbrace{1,...,1}_{2q_1}, \ldots, \underbrace{j, \ldots, j}_{2q_j}, \ldots,\underbrace{r-1, \ldots, r-1}_{2q_{r-1}},j,j].$$
Its stabilizer with respect to the absolute conjugation does not depend on $j$: it is the subgroup of $G(r,2m-1)$ given by
$$H_{q_0,\ldots,q_{r-1}}=\{g \in G:|g| \in S_{2q_0}\times \ldots \times S_{2q_{r-1}}, |g|(i+1)=|g|(i)\pm 1\, \forall \,i
\mbox{ odd, } 0< i< 2m\}.$$
Thus, our module can be written as
$$\bigoplus_{q_0, \ldots,q_{r-1 }=m-1}\bigoplus_{j=o}^{r-1}M_{q_0,\ldots,q_j+1, \ldots,q_{r-1}}^j=
\bigoplus_{q_0+\cdots+q_{r-1}=m-1}\bigoplus_{j=o}^{r-1}(\mathbb{C}\,v_{q_0,\ldots,q_{r-1}}^j)\big \uparrow_{H_{q_0,\ldots,q_{r-1}}}^{G(r, 2m-1)}.$$

As for the right side of the isomorphism, we have
$$
M_{m-1,r}\uparrow^{G(r,2m-1)}=\bigoplus_{q_0+\cdots+q_{r-1}=m-1}M_{q_0, \ldots,q_{r-1}}\big\uparrow^{G(r,2m-1)}.
$$
We choose this time
$$u_{q_0, \ldots,q_{r-1}}\eqdef[(2,1,4,3,...,2m-2,2m-3);
\underbrace{0,0,...0}_{2q_0},\underbrace{1,...,1}_{2q_1},  \ldots,\underbrace{r-1, \ldots, r-1}_{2q_{r-1}}],$$
whose stabilizer with respect to the absolute conjugation in $G(r, 2(m-1))$ is
$$K_{q_0,\ldots,q_{r-1}}=\{g \in G:|g| \in S_{2q_0}\times \ldots \times S_{2q_{r-1}}, |g|(i+1)=|g|(i)\pm 1\, \forall \,i
\mbox{ odd, } 0< i< 2m-2\}.$$
We observe that $K_{q_0,\ldots,q_{r-1}}$ is a subgroup of index $r$ in $H_{q_0,\ldots,q_{r-1}}$, and a system of coset representatives is given by $$C=\{\sigma_i\eqdef[\Id;\underbrace{0,...,0}_{2(m-1)},i,0]\}_{i=0,\ldots, r-1}.$$

So we can split the induction into two steps, and we get
\begin{align*}
M_{m-1,r}\uparrow^{G(r,2m-1)}&=\bigoplus_{q_0+\cdots+q_{r-1}=m-1}M_{q_0, \ldots,q_{r-1}}\big\uparrow^{G(r,2m-1)}\\
&=
\bigoplus_{q_0+\cdots+q_{r-1}=m-1}
(\mathbb{C}\,u_{q_0,\ldots,q_{r-1}})\big \uparrow_{K_{q_0,\ldots,q_{r-1}}}^{G(r,2m-1)}\\
&=
\bigoplus_{q_0+\cdots+q_{r-1}=m-1}
\left((\mathbb{C}\,u_{q_0,\ldots,q_{r-1}})\big\uparrow_{K_{q_0,\ldots,q_{r-1}}}^{H_{q_0,\ldots,q_{r-1}}}\right)\Big\uparrow_{H_{q_0,\ldots,
q_{r-1}}}^{G(r,2m-1)}
\end{align*}

So we are enquiring if
\begin{align*}\bigoplus_{q_0+\cdots+q_{r-1}=m-1}&\bigoplus_{j=o}^{r-1}\left(\mathbb{C}\,v_{q_0,\ldots,q_{r-1}}^j\right) \big\uparrow_{H_{q_0,\ldots,q_{r-1}}}^{G(r, 2m-1)}\cong\\
&\bigoplus_{q_0+\cdots+q_{r-1}=m-1}
\left((\mathbb{C}\,u_{q_0,\ldots,q_{r-1}})\big\uparrow_{K_{q_0,\ldots,q_{r-1}}}^{H_{q_0,\ldots,q_{r-1}}}\right)\Big\uparrow_{H_{q_0,\ldots,
q_{r-1}}}^{G(r,2m-1)},
\end{align*}
and all we need to show is that
$$\bigoplus_{j=o}^{r-1}\mathbb{C}\,v_{q_0,\ldots,q_{r-1}}^j \cong(\mathbb{C}\,u_{q_0,\ldots,q_{r-1}})\big\uparrow_{K_{q_0,\ldots,q_{r-1}}}^{H_{q_0,\ldots,q_{r-1}}}$$
as ${H_{q_0,\ldots,q_{r-1}}}$-modules.

Let us compute characters. The character $\chi_1$ of the representation on the left is given by
\begin{align*}
\chi_1(g)&=\sum_{j=0}^{r-1}\zeta_r^{<g,v_{q_0,\ldots,q_{r-1}}^j>}=\zeta_r^{<g,u_{q_0,\ldots,q_{r-1}}>}
\sum_{j=0}^{ r-1}\zeta_r^{jz_{2m-1(g)}}\\
&=\left\{
\begin{array}{ll}0 &\mbox{ if }z_{2m-1}(g)\neq 0;\\
 r \zeta_r^{<g,u_{q_0,\ldots,q_{r-1}}>} &\mbox{ if } z_{2m-1}(g)= 0.                                                                                                              \end{array} \right.
\end{align*}

As for the character $\chi_2$ of the representation on the right, we have
\begin{align*}
\chi_2(g)&=\sum_{\begin{subarray}{c}h\in C\\h^{-1}gh \in B_{q_0,\ldots,q_{r-1}}\end{subarray} }\chi(h^{-1}gh)\\&=\left\{
\begin{array}{ll}0 \qquad &\mbox{ if }z_{2m-1}(g)\neq 0;\\
 r \zeta_r^{<g,u_{q_0,\ldots,q_{r-1}}>}&\mbox{ if } z_{2m-1}(g)= 0,                                                                                                              \end{array} \right.
\end{align*}
so the two characters agree and the representations are isomorphic.

So we know that the modules $M_{m.r}$ satisfy the conditions of Proposition \ref{grnindres} and to complete the proof of Theorem \ref{nofixedpointsG(r,n)}, 
%
%
%
generalizing what was done for $B_n$, it suffices to show that there exist representations $\sigma_0$  of $S_{2p_{0}}, \ldots,\sigma_{r-1}$ of  $S_{2p_{r-1}}$ such that
\begin{equation}\label{sita2}
(M_{p_{0}, \ldots, p_{r-1}},\varphi )\cong\Ind_{G(r,2p_0)\times \ldots G(r,2p_{r-1})}^{G(r,2m)}(\tilde \sigma_0 \odot (\gamma_{2(p_{1})} \otimes \tilde \sigma_1)\odot \dots \odot (\gamma_{2(p_{r-1})}^{r-1} \otimes \tilde \sigma_{r-1}) ),
\end{equation}
where the $\tilde \sigma_i$'s are the natural extensions of $\sigma_i$ to $G(r,2p_{i})$.

If we set $S_i\eqdef[p_0+\cdots+p_{i-1}+1,p_0+\cdots+p_{i-1}+p_{i}]$, we
consider the vector space 
$M_{S_0,\ldots,S_{r-1}}\eqdef\Span\{C_v:v\in \Delta_{S_0, \ldots, S_{r-1}}\}$.
We have
$$M_{p_{0}, \ldots, p_{r-1}}=M_{S_0,\ldots,S_{r-1}}\uparrow_{G(r,2p_0)\times \cdots \times G(r,2p_{r-1})}^{G(r,2m)}.$$
Let us define $M_i\eqdef \Span \{C_{v_i}:v_i \mbox{ is an involution in }S_{2p_i}\}$. Then
\begin{align*}
M_{S_0,\ldots,S_{r-1}}&\cong M_0 \times \cdots \times M_{r-1}\\
C_{v_0, \ldots, v_{r-1}}&\mapsto C_{v_0}\otimes \zeta_rC_{v_1} \otimes\cdots \otimes \zeta_r^{r-1} C_{v_{r-1}}
\end{align*}
Arguing as for $B_n$, let $g=g_0,g_1, \ldots, g_{r-1}\in {G(r,2p_0)\times \cdots \times G(r,2p_{r-1})}$. We get \begin{eqnarray*}
&\varphi(g)C_{v_0}\otimes \cdots \otimes C_{v_{r-1}}\leftrightarrow \varphi(g)C_v
= (\zeta_r)^{<g,v>}C_{|g|v|g|^{-1}}\\
&\leftrightarrow C_{|g_0|v_0|g_0|^{-1}}\otimes (\zeta_r)^{z(g_1)}C_{|g_1|v_1|g_1|^{-1}} \otimes \cdots \otimes  (\zeta_r)^{(r-1)z(g_{r-1})}C_{|g_{r-1}|v_{r-1}|g_{r-1}|^{-1}}
\end{eqnarray*}
and Equation \eqref{sita2} is achieved. Our claim follows from the irreducible decomposition of the representations $\sigma_i$, the 
ption of the irreducible representations of $G(r,n)$ in Proposition \ref{rapp di grn}, and Theorem \ref{nofixedpointsG(r,n)}.

Before leaving the module $M_{m,r}$ with no fixed points and going on to study the decomposition of the whole model $M$, we only need to show that stepping from $\varphi$ to $\varrho$ is just like exchanging rows and columns.
Up to obvious modifications, this result can be attained just as it was done in the case of $B_n$, so we will not treat it.

Summing up, at this point we can give for granted that:
\begin{equation}\label{raffinato}
(M_{p_{0},\ldots, p_{r-1}}, \varrho)\cong\bigoplus_{\substack{
                                                            |\lambda_i|=2p_{i}\\
                                                            \lambda_i
                                                     \textrm{ with
                                                            no odd
                                                            columns}
                                                            }}
 \rho_{\lambda_0, \ldots,\lambda_{r-1}}.
\end{equation}

Let us take a step forward towards the proof of Theorem \ref{mainth}: we are now dealing with the modules $M(c_{f_0,\ldots,f_{r-1}, p_0,\ldots,p_{r-1}})$, where $f_0+\ldots+f_{r-1}+2 p_0,\ldots+2p_{r-1} =n$. Let $p_0+\ldots+p_{r-1}=m$ and let us consider the $G(r, 2m)$-module  $$\Pi_{p_0,\ldots, p_{r-1}}\eqdef\big(M_{p_0,\ldots, p_{r-1}}, \varrho\big).$$ We know its irreducible decomposition from \eqref{raffinato}.
Arguing as above, we can infer that
\begin{equation} \label{pappappero2}
(M(c_{f_{0}, \ldots f_{r-1}, p_{0}, \ldots p'_{r-1}}),\varrho)\cong \Ind_{G(r,2m)\times G(r,n-2m)}^{G(r,n)}(\Pi_{m,r}^{p_0,\ldots, p_{r-1}}\odot \rho_{\iota_{f_0},\ldots,\iota_{f_{r-1}}}),
\end{equation}
and Theorem \ref{mainth} follows from the $G(r,n)$-version of Pieri rule.

\emph{E-mail address: }{\tt caselli,fulci@dm.unibo.it}
\end{document}